\documentclass{article}
\usepackage{amsfonts,amsmath,amsthm,amssymb}
\usepackage{graphics, epsfig}
\usepackage{color}
\usepackage{appendix}
\usepackage{ulem}
\usepackage[makeroom]{cancel}
\usepackage{fancyhdr}
\usepackage{centernot}
\usepackage{tikz-cd}
\usepackage{mathtools}
\usepackage{ stmaryrd }

 \usepackage[usenames,dvipsnames]{pstricks}
 \usepackage{pst-grad} 
 \usepackage{pst-plot} 
\allowdisplaybreaks

\let\TeXchi\chi
\newbox\chibox
\setbox0 \hbox{\mathsurround0pt $\TeXchi$}
\setbox\chibox \hbox{\raise\dp0 \box 0 }
\def\chi{\copy\chibox}

\newtheorem{proposition}{Proposition}[section]
\newtheorem{theorem}{Theorem}[section]

\newtheorem{lemma}{Lemma}[section]
\newtheorem{corollary}{Corollary}[section]

\numberwithin{equation}{section}
\numberwithin{theorem}{section}
\numberwithin{definition}{section}
\numberwithin{example}{section}
\numberwithin{proposition}{section}
\numberwithin{lemma}{section}
\numberwithin{remark}{section}
\DeclareMathOperator{\md}{mod}
\DeclareMathOperator{\ord}{ord}
\newcommand\blfootnote[1]{%
  \begingroup
  \renewcommand\thefootnote{}\footnote{#1}%
  \addtocounter{footnote}{-1}%
  \endgroup
}
\begin{document}
\title{A group-theoretical approach to Lehmer's totient problem}
\author
{Manuel Norman}
\date{}
\maketitle
\begin{abstract}
\noindent Lehmer's totient problem asks whether there exists any composite number $n$ such that $\varphi(n) \, \mid \, (n-1)$, where $\varphi$ is Euler totient function. It is known that if any such $n$ exists, it must be Carmichael and $n > 10^{30}$. In this paper, we develop a new approach to the problem via some recent results in group theory related to a function $\psi$ (the sum of order of elements of a group) and show that if $k \varphi(n) = n-1$ for some integer $k$, then $k$ must be $\geq 3$, and actually, if $5, 7, 11, 13 \not | n$, $k \geq 4$. This implies that any counterexample must be such that $n > 10^{8171}$ and $\omega(n) \geq 1991$.
\end{abstract}
\blfootnote{Author: \textbf{Manuel Norman}; email: manuel.norman02@gmail.com\\
\textbf{AMS Subject Classification (2020)}: 11A25, 20D60\\
\textbf{Key Words}: Lehmer's totient problem, group, order}
\section{Introduction}
Lehmer's totient problem is an important unsolved question in number theory, which asks whether the well known property $\varphi(n)=n-1 \Leftrightarrow$ $n$ is prime can be generalised to $\varphi(n) | (n-1) \Leftrightarrow $ $n$ is prime. Here, $\varphi$ denotes Euler totient function, that is, the multiplicative arithmetic function defined by:
\begin{equation}\label{Eq:1.1}
\phi(n):= n \prod_{p | n} \left( 1 - \frac{1}{p} \right)
\end{equation}
The problem has been studied by many mathematicians (we refer to [1-5] for some results on this topic), and up to now it is known that if any counterexample exists, then it must be bigger than $10^{30}$, and it must be Carmichael (and hence odd and square-free). We recall that a composite number $n$ is called Carmichael if
$$ b^n \equiv b \, (\md n) $$
for every integer $b$. The following result, often called Korselt's criterion, gives us another way to recognise a Carmichael number:
\begin{proposition}\label{Prop:1.1}
A composite number $n$ is Carmichael if and only if it is square-free and $(p-1) | (n-1)$ for every prime $p \, | \, n$.
\end{proposition}
More recently, the problem has been approached by studying some special sequences for which it is proved that no Lehmer number (i.e. a number which does not satisfy Lehmer's totient problem) belongs to them. We refer to [14-19] for some results in this direction.\\
The aim of this paper is to consider the equation
\begin{equation}\label{Eq:1.2}
k \varphi(n) = n-1
\end{equation}
which, assuming a counterexample to Lehmer's problem exists, would hold true for some integer $k \geq 2$ (note that $k \neq 1$ because otherwise $n$ would certainly be prime). As noticed, for instance, in [2], an improvement on the value of $k$ would lead to much better estimates on $n$: for example, if $k \geq 3$, then $n > 10^{8171}$ and $\omega(n) \geq 1991$. The aim of this paper is to prove the following result:
\begin{theorem}\label{Thm:1.1}
Let $n$ be a composite number satisfying \eqref{Eq:1.2} for some integer $k \geq 2$. Then:\\
(i) if $3|n$, $k \geq 4$ and $k \equiv 1 \, (\md 3)$;\\
(ii) if $3 \not | n$, $k \geq 3$;\\
(iii) if $3, 5, 7, 11, 13 \not | n$ , then $k \geq 4$.
\end{theorem}
In particular, from this result we obtain, by what we said before:
\begin{corollary}\label{Crl:1.1}
Let $n$ be a composite number satisfying \eqref{Eq:1.2}. Then, $k \geq 3$ and hence $n > 10^{8171}$, $\omega(n) \geq 1991$.
\end{corollary}
Moreover, we also show that the abundancy index of such a number must satisfy a certain inequality:
\begin{proposition}\label{Prop:1.2}
If a composite $n$ satisfies \eqref{Eq:1.2}, then:
$$ I(n):= \frac{\sigma(n)}{n} > \frac{24}{\pi^2} \approx 2.431708... $$
Furthermore, if $3, 5, 7, 11, 13 \not | n$:
$$ I(n) > \frac{715715}{18432 \pi^2} \approx 3.9343... $$
\end{proposition}
Here, $\sigma$ is the sum of divisors function.\\
In order to show these results, we develop a new approach to the problem building upon some recent papers on group theory. In the next section we will recall all the tools which will be needed in our proofs, while our main theorem is proven in Section 3.
\section{Preliminaries}
In this section we recall some important results which will be used in our new method to approach Lehmer's totient problem. First of all, for a group $G$ we define the following function:
\begin{equation}\label{Eq:2.1}
\psi(G):= \sum_{x \in G} \ord(x)
\end{equation}
The properties of this function have been studied in the last years, obtaining many interesting results which allows one to classify many kinds of group depending on the values of $\psi$ or of two related functions, namely $\psi'$ and $\psi''$, defined as follows:
\begin{equation}\label{Eq:2.2}
\psi'(G):= \frac{\psi(G)}{\psi(C_{|G|})}
\end{equation}
and
\begin{equation}\label{Eq:2.3}
\psi''(G):= \frac{\psi(G)}{|G|^2}
\end{equation}
Here, $|G|$ denotes the order of $G$ and $C_m$ is the cyclic group of order $m$. The reason why these two quantities are of particular interest is given in the next two results:
\begin{proposition}\label{Prop:2.1}
Let $G$ be a group of order $n$. Then:
$$ \psi(G) \leq \psi(C_n) $$
Moreover, equality holds if and only if $G$ is cyclic (and hence isomorphic to $C_n$).
\end{proposition}
\begin{proposition}\label{Prop:2.2}
Let $G$ be a group of order $n$. Then:
$$ \psi(G) \leq n^2 $$
\end{proposition}
\begin{proof}
As it is well known, Gauss proved that:
$$ \sum_{d | n} \phi(d) =n $$
Consequently, by Proposition \ref{Eq:2.1} and the fact that
$$ \psi(C_n) = \sum_{d | n} d\phi(d)  $$
we have that:
$$ \psi(G) \leq n^2 $$
and the result follows.
\end{proof}
We will also recall that:
\begin{proposition}\label{Prop:2.3}
$\psi$ is a multiplicative function.
\end{proposition}
Now, our new idea to approach Lehmer's problem relies on some results which relate $\psi(C_n)$ with $\psi(G)$ for some other groups of the same order. Among these ones, we collect in one theorem the following:
\begin{theorem}\label{Thm:2.1}
Let $G$ be a noncyclic group of order $n$. Then:\\
(i) $\psi(G) \leq \frac{7}{11} \psi(C_n)$, and equality holds if and only if $n=4m$, $m$ odd, and $G=C_2 \times C_2 \times C_m$;\\
(ii) If $q$ is the smallest prime factor of $n$, the following holds, and equality holds if and only if $n=q^2 r$, $\gcd(r,q!)=1$ and $G=C_q \times C_q \times C_r$:
$$ \psi(G) \leq \frac{((q^2 -1)q +1)(q+1)}{q^5 +1} \psi(C_n)$$
(iii) If $n=2m$, $m$ odd:
$$ \psi(G) \leq \frac{13}{21} \psi(C_n) $$
(iv) If $n=8m$, $m$ odd:
$$ \psi(G) \leq \frac{27}{43} \psi(C_n) $$
and equality holds if and only if $G=Q_8 \times C_m$, where $Q_8$ is the quaternion group;\\
(v) If $n= 2^{\alpha} m$, $m$ odd and $\alpha \geq 4$:
$$ \psi(G) \leq \frac{2^{2 \alpha +3}+7}{7(1+2^{2 \alpha +1})} $$
(vi) If $n=2m$ with $m=p_1^{\alpha_1} \cdot \cdot \cdot p_j^{\alpha_j}$ odd, and $l:= \min \lbrace p_i^{\alpha_i} \rbrace$, then:
$$ \psi(G) \leq \left( \frac{1}{3} + \frac{2l}{\psi(C_l)} \right) \psi(C_n) $$
and equality holds if and only if $G=D_{2l} \times C_{m/l}$.
\end{theorem}
We also state here a useful result on $\psi''$, which gives us the possibility to determine the nature of the group we are dealing with:
\begin{theorem}\label{Thm:2.2}
Let $G$ be a group. Then:\\
(i) If $\psi''(G) > \frac{7}{16}$, then $G$ is cyclic;\\
(ii) If $\psi''(G) > \frac{27}{64}$, then $G$ is abelian;\\
(iii) If $\psi''(G) > \frac{13}{36}$, then $G$ is nilpotent;\\
(iv) If $\psi''(G) > \frac{31}{144}$, then $G$ is supersolvable;\\  
(v) If $\psi''(G) > \frac{211}{3600}$, then $G$ is solvable.
\end{theorem}
However, this results is not the best possible, since we would be interested in the converse problem: for instance, if $G$ is abelian, what can we say about $\psi''(G)$? This is indeed an open problem (see [7]), and there has been some progress recently. We refer to [7-12] for the above results and more details on this topic. We now report here a fundamental result from [8], which will be crucial (in a slightly modified version) in our method:
\begin{theorem}\label{Thm:2.3}
A group $G$ is nilpotent if and only if
$$ \psi(G) \geq \prod_{p | n} p(p^{\alpha_p} -1) +1 $$
where $|G|=n= \prod_{p|n} p^{\alpha_p}$. Moreover, equality holds if and only if all the Sylow subgroups of $G$ have prime exponent.
\end{theorem}
\begin{lemma}\label{Lm:2.1}
If $G$ is a noncyclic, nilpotent group of order $4n$, with $n$ squarefree and odd, then:
$$ \psi''(G) > \frac{\varphi(n)}{2n} $$
\end{lemma}
\begin{proof}
The proof is essentially the same as the one of Theorem \ref{Thm:2.3}. We can also prove it as follows. By using Theorem 2.2 in [20] and Theorem \ref{Thm:2.3} above, we can easily notice that:
$$ \psi(G) = \psi(V) \psi(U) $$
where $V$ is the $2$-Sylow subgroup of $G$ of order $4$ and $U$ is the direct product of all the other Sylow subgroups of $G$ (recall: since $G$ is nilpotent, it is equal to the direct product of all its Sylow subgroups). Then, since $\psi$ is multiplicative, the above equality holds. By Theorem \ref{Thm:2.3}, we derive:
$$ \psi''(U) > \frac{\varphi(n)}{n} $$
while, by Theorem 2.2 in [20], we get:
$$ \psi''(V) \geq \frac{2^2 \cdot 2}{16}= \frac{1}{2} $$
Thus, the conclusion follows.
\end{proof}
We are now almost ready to introduce our new method. The last result that we need before proceeding involves the existence of nilpotent groups of a certain order. We recall that a number $n$ is called cyclic if and only if $\gcd(n, \varphi(n))=1$, that is, if and only if every group of order $n$ is cyclic. The next Proposition gives us a simple criterion to determine when, for some number $n$, there exists at least one noncylic nilpotent group of order $n$ (this will be useful in view of the previous result: since it is not yet known a complete converse to Theorem \ref{Thm:2.2}, we will make use of nilpotent groups, for which instead the conclusion holds by Theorem \ref{Thm:2.3}):
\begin{proposition}\label{Prop:2.4}
If $n$ is odd, then there exists a nonclyclic nilpotent groups of order $4n$.
\end{proposition}
\begin{proof}
We will rely on some basic results in group theory, see for instance [13]. We recall that:\\
(i) the direct product of nilpotent groups is nilpotent;\\
(ii) every finite $p$-group is nilpotent;\\
(iii) a nilpotent groups $G$ is cyclic iff all its $p$-Sylow subgroups are cyclic;\\
(iv) a finite group $V$ is a $p$-group iff $|V|=p^a$ for some positive integer $a$.\\
Let $4n=4 \cdot p_1 ^{\alpha_1} \cdot \cdot \cdot p_j^{\alpha_j}$ be the prime factorisation of $4n$. Then, by (iv), if we take any group $V_i$ of order $p_i ^{\alpha_i}$ for each $i$, every such group is a $p_i$-group (for the considered $i$, of course). Then, by (ii) and (i), we conclude that their direct product is nilpotent.\\
Now let $V$ be a nonclycic group of order $4$ (since $4$ is not a cyclic number, because $\gcd(\varphi(4),4)=2 \neq 1$, such a group certainly exists). The direct product of all the $V_i$'s and $V$ is a nilpotent group (again by (iv), (ii) and (i)), and hence, since $V$ is not cyclic, by (iii), we can conclude that $G$, defined as the direct product of these groups, is noncyclic and nilpotent. The proof is concluded.
\end{proof}
Via this proposition, we can say that if $n$ satisfies \eqref{Eq:1.2}, then there exists a nonclyclic nilpotent group of order $4n$. This will give us the possibility to apply the previous results and get new properties of Lehmer's numbers.
\section{Main results}
We can finally begin to develop our new method to study Lehmer's totient problem. Consider a noncyclic nilpotent group of order $4n$, with $n$ satisfying \eqref{Eq:1.2}. By Proposition \ref{Prop:2.4}, such a group certainly exists. By applying (i) of Theorem \ref{Thm:2.1}, we obtain:
$$ \psi(G) \leq \frac{7}{11} \psi(C_{4n}) $$
Since it can be easily seen that
$$ \psi(C_m)= \prod_{p | m} \frac{p^{2 \alpha_p +1} +1}{p+1} $$
(see the references for such result), by assumption we obtain:
$$ \psi(G) \leq \frac{7}{11} \psi(C_{4n}) = 7 \psi(G) = 7 \left( \frac{n(n-1)}{k} + \sum_{d|n, d \neq n} d \varphi(d) \right) $$
Since $n$ is odd, $d \leq \frac{n}{3}$, because $d \neq n$, so:
$$ \psi(G) \leq 7 \left( \frac{n(n-1)}{k} + \frac{n}{3} \sum_{d|n, d \neq n}  \varphi(d) \right) = 7 \left( \frac{n(n-1)}{k} + \frac{n}{3} \left( n - \frac{(n-1)}{k} \right) \right) $$ 
It is then clear that:
$$ \psi(G) \leq 7 \left( \frac{2n^2}{3k} + \frac{n^2}{3} \right) $$
Suppose that there exists a noncyclic nilpotent group $G$ of order $4n$ such that:
$$ \psi''(G) > \frac{7}{24} $$
Then, by what we said above, we get:
$$ \frac{7}{24} < \psi''(G) \leq \frac{7}{16} \left( \frac{2}{3k} + \frac{1}{3} \right) $$
Solving this inequality with respect to $k$, we obtain:
$$ k < 2 $$
But then, $n$ does not satisfy \eqref{Eq:1.2}, a contradiction. We thus conclude that, if a Carmichael number $n$ is such that there exists a noncyclic nilpotent group of order $4n$ with $\psi''$ bigger than $\frac{7}{24}$, then $n$ is not a counterexample to Lehmer's totient problem.\\
The main idea of our method is essentially this one. Now we need to prove the existence of some group as above, and this is where the results on nilpotent groups will come into play. But before doing this, we notice that the above result can be generalised in a first way as follows: we only need to use Theorem \ref{Thm:2.1} with groups of order $2^{\alpha} n$ to prove
\begin{theorem}\label{Thm:3.1}
Let $n$ be a Carmichael number. If any of the following group exists, then $n$ does not satisfy \eqref{Eq:1.2}.\\
(i) There exists a noncyclic group $G$ of order $2n$ such that $\psi''(G) > \frac{13}{42}$;\\
(ii) There exists a noncyclic group $G$ of order $4n$ such that $\psi''(G) > \frac{7}{24}$;\\
(iii) There exists a noncyclic group $G$ of order $8n$ such that $\psi''(G) > \frac{9}{32}$;\\
(iv) There exists a noncyclic group of order $2^{\alpha} n$, $\alpha \geq 4$, such that
$$ \psi''(G) > \frac{16}{63} + \frac{1}{9 \cdot 2^{2 \alpha -1}} $$
\end{theorem}
We notice that unfortunately, (i) cannot be used together with Theorem \ref{Thm:2.3}. If this theorem could be replaced by some other (suitable) converse of Theorem \ref{Thm:2.2}, it would highly improve our bounds, because of (vi) in Theorem \ref{Thm:2.1}. Even an extension of (vi) to groups of order $4n$, $n$ odd, would guarantee some refinements of the bounds.\\
The above result will be often improved later, in order to derive better results. We start with the following:
\begin{proposition}\label{Prop:3.1}
Let $n$ be Carmichael, and let $q$ be its smallest prime factor. If there exists a group $G$ of order $4n$ such that:
$$ \psi''(G) > \frac{7}{16} \left( \frac{q-1}{2q} + \frac{1}{q} \right) $$
then $n$ does not satisfy \eqref{Eq:1.2}.
\end{proposition}
\begin{proof}
The idea of the proof is the same as before. Here, however, we use the assumption that $q$ is the smallest prime factor of $n$ to get a better estimate than the previous one. Proceeding as before:
$$ \psi(G) \leq 7 \cdot \left( \frac{n(n-1)}{k} + \frac{n}{q} \sum_{d|n, d \neq n}  \varphi(d) \right) =  7 \left( \frac{n(n-1)}{k} + \frac{n}{q} \left( n - \frac{(n-1)}{k} \right) \right) \leq  $$
$$ \leq 7 \left( \frac{(q-1)n^2}{qk} + \frac{n^2}{q} \right) $$
Now, by assumption, we get:
$$ \frac{7}{16} \left( \frac{q-1}{2q} + \frac{1}{q} \right)  < \frac{7}{16}  \left( \frac{q-1}{qk} + \frac{1}{q} \right) $$
But then:
$$ \frac{q-1}{2q} + \frac{1}{q} < \frac{q-1}{kq} + \frac{1}{q} $$
from which we derive:
$$ k < 2 $$
As before, the result follows.
\end{proof}
We now begin to study $\psi''$ in order to obtain the desidered lower bounds, at least under some conditions. Since by Proposition \ref{Prop:2.4} there is always at least one nilpotent noncyclic group of order $4n$ for $n$ odd, let $G$ be any such a group. The slight improvement of Theorem \ref{Thm:2.3}, namely Lemma \ref{Lm:2.1}, applies, and hence:
$$ \psi''(G) \geq \frac{n \varphi(n)}{2n^2} > \frac{\varphi(n)}{2n} $$
If we assume that Lehmer's totient problem does not hold for $n$, and so it satisfies \eqref{Eq:1.2}, we have $n > 10^{30}$, from which we can deduce that:
$$ \psi''(G) > \frac{n(n-1)}{2kn} > \frac{1}{2k} \left(1 - \frac{1}{10^{30}} \right) $$
Suppose $k =2$. Then, for $q \geq 11$:
$$ \psi''(G) > \frac{1}{4} \left( 1 - \frac{1}{10^{30}} \right) > \frac{7}{16} \left( \frac{q-1}{2q} + \frac{1}{q} \right) $$
But by Proposition \ref{Prop:3.1}, this implies that $n$ does not satisfy \eqref{Eq:1.2}, a contradiction. Hence, if any Carmichael number $n$ satisfies \eqref{Eq:1.2} and is not divisible by $3$, $5$ and $7$, $k \neq 2$. Furthermore, as noticed, for instance, in [3], if $3 |n$, then $k \equiv 1 (\md 3)$, which implies $k \geq 4$ for $n$ divisible by $3$. As a consequence:
\begin{proposition}\label{Prop:3.2}
If $n$ is a solution to \eqref{Eq:1.2}, and neither $5$ nor $7$ divides $n$, then $k \geq 3$.
\end{proposition}
By [2], it immediately follows that any counterexample to Lehmer's totient problem which is not divisible by $5$ and $7$ must satisfy $n > 10^{8171}$ and $\omega(n) \geq 1991$.\\
We actually notice that this result can be proven in a slightly different way, obtaining even better bounds. As we said above, we can suppose $q \geq 5$, since the case $3 |n$ has already been settled. It is well known that the following inequality holds:
\begin{proposition}\label{Prop:3.3}
\begin{equation}\label{Eq:3.1}
1 > \frac{\varphi(n) \sigma(n)}{n^2} > \frac{6}{\pi^2}
\end{equation}
\end{proposition}
In some special cases, the lower bound can be improved. For instance, for $n$ Carmichael and not divisible by $3$, we easily see that:
$$ 1 > \prod_{p|n} \left( 1 - \frac{1}{p^2} \right) > \prod_{p \geq 5} \left( 1 - \frac{1}{p^2} \right)= \frac{6}{\pi^2} \cdot \frac{4}{3} \cdot \frac{9}{8} = \frac{9}{\pi^2} $$
Now, the assumption that $k = 2$ implies $\sigma(n) < 2.0001 n$ (say). Indeed, if this were not true we would have, by the upper bound in Proposition \ref{Prop:3.3}:
$$ 2=k \geq 2.0001 \cdot \left( 1 - \frac{1}{10^{30}} \right) > 2$$
absurd. Therefore, by Lemma \ref{Lm:2.1} we can conclude that, for some noncyclic nilpotent group $G$ of order $4n$:
$$ \psi''(G) > \frac{\varphi(n)}{2n} > \frac{9}{2 \pi^2 \cdot 2.0001} $$
Some simple calculations show that this value is bigger than $\frac{7}{16} \left( \frac{q-1}{2q} + \frac{1}{q} \right)$ for $q \geq 11$. But then, we can improve the lower bound by multiplying by $\frac{25}{24}$, assuming that $5$ does not divide $n$. This way, we obtain:
$$ \psi''(G) > \frac{75}{8 \pi^2 \cdot 2.0001} $$
which is easily seen to be bigger than $\frac{7}{16} \left( \frac{q-1}{2q} + \frac{1}{q} \right)$ for $q \geq 7$. Thus, this method gives us an improvement of the previous bound, so that $k \geq 3$, except possibly when $5 | n$.\\
Now we proceed by improving, again, our previous result. The proof is almost the same as before, so we do not explicitely write it.
\begin{proposition}\label{Prop:3.4}
Let $n$ be Carmichael, with $q$ its smallest prime factor, and let $R \geq 2$ be a natural number. Then, if there exists a group $G$ of order $2n$ such that:
$$ \psi''(G) > \frac{7}{16} \left( \frac{q-1}{Rq} + \frac{1}{q} \right) $$
and $n$ satisfies \eqref{Eq:1.2}, $k \leq R-1$.
\end{proposition}
\begin{corollary}\label{Crl:3.1}
Let $n$ be Carmichael, $5 \not | n$, with $q$ its smallest prime factor. Then, if there exists a noncyluc group $G$ of order $4n$ such that:
$$ \psi''(G) > \frac{7}{16} \left( \frac{q-1}{3q} + \frac{1}{q} \right) $$
$n$ does not satisfy \eqref{Eq:1.2}.
\end{corollary}
\begin{proof}
This follows from Proposition \ref{Prop:3.4} together with Proposition \ref{Prop:3.2} and the extension proved below it.
\end{proof}
Hence, we can apply the same idea as above. Again by Lemma \ref{Lm:2.1}, assuming $k=3$:
$$ \psi''(G) > \frac{\varphi(n)}{2n} > \frac{1}{6} \cdot \left( 1 - \frac{1}{10^{8171}} \right) > \frac{7}{16} \left( \frac{q-1}{3q} + \frac{1}{q} \right) $$
for $q \geq 17$. Thus, by Corollary \ref{Crl:3.1}, for $q \geq 17$ (and $q=3$), $k \geq 4$. \\
We now deal with the case $q=5$. Suppose that $5 | n$ and $3 \not | n$, $7 \not | n$ (if $3|n$, we are done; if $7|n$, we have another case, which will be considered later). We proceed as in the proof of Proposition \ref{Prop:3.1}, but with a slight modification. We have:
$$ \psi(G) \leq 7 \cdot \left( \frac{n(n-1)}{k} + \frac{n(n-1)}{5k \cdot \varphi(5)} + \frac{n}{11} \sum_{d|n, d \neq n, n/5}  \varphi(d) \right) = $$
$$= 7 \left( \frac{n(n-1)}{k}  + \frac{n(n-1)}{20k} + \frac{n}{11} \left( n - \frac{(n-1)}{k} - \frac{n(n-1)}{20k} \right) \right) \leq  $$
$$ \leq 7 \left( \frac{21n^2}{22k} + \frac{n^2}{11} \right) $$
If there exists a noncyclic nilpotent group $G$ of order $4n$ such that
$$ \psi''(G) > \frac{175}{704} $$
then, by the above inequality:
$$ k < 2 $$
so that, by Proposition  \ref{Prop:3.1}, $n$ does not satisfy \eqref{Eq:1.2} and we are done. But clearly:
$$ \frac{1}{4} \left( 1 - \frac{1}{10^{30}} \right) > \frac{175}{704} $$
so that, if $5 | n$ and $7 \not | n$, $k \geq 3$.\\
If instead $5|n$, $7|n$ and $3 \not | n$, $13 \not |n$, we get, similarly (noting that $n$ Carmichael and $5|n$ implies $11 \not | n$, by [1]):
$$ \psi(G) \leq 7 \cdot \left( \frac{n(n-1)}{k} + \frac{n(n-1)}{5k \cdot \varphi(5)} + \frac{n(n-1)}{7k \cdot \varphi(7)} + \frac{n}{17} \sum_{d|n, d \neq n, n/5, n/7}  \varphi(d) \right) = $$
$$= 7 \left( \frac{n(n-1)}{k}  + \frac{n(n-1)}{20k} + \frac{n(n-1)}{42k} + \frac{n}{17} \left( n - \frac{(n-1)}{k} - \frac{n(n-1)}{20k} - \frac{n(n-1)}{42k} \right) \right) \leq  $$
$$ \leq 7 \left( \frac{1804n^2}{1785k} + \frac{n^2}{17} \right) $$
If there exists a noncyclic nilpotent group $G$ of order $4n$ such that
$$ \psi''(G) > \frac{1007}{4080} $$
then, by the above inequality:
$$ k < 2 $$
so that, by Proposition  \ref{Prop:3.1}, $n$ does not satisfy \eqref{Eq:1.2} and we are done. But again, this value is less than $\frac{1}{4} \left( 1 - \frac{1}{10^{30}} \right)$, and hence we only need to consider the case $5, 7, 13 | n$, $3 \not | n$. But, this time, by using the same method as above, we finally get that $k \geq 3$, without imposing other restrictions (i.e. other divisors of $n$). Hence, $k \geq 3$ whenever $n$ satisfies \eqref{Eq:1.2}.\\
Consequently, we have shown that:
\begin{proposition}\label{Prop:3.5}
If $n$ satisfies \eqref{Eq:1.2}, then $k \geq 3$. If $q=3$ or $q \geq 17$, moreover, $k \geq 4$
\end{proposition} 
We can now proceed as before. The following result is straightforward:
\begin{theorem}\label{Thm:3.2}
Suppose that $n$ satisfies \eqref{Eq:1.2}. If $q \geq 17$ is the smallest prime factor of $n$ and $R \geq 4$ is an integer, then, assuming that:
$$ \frac{1}{2} R \left( \frac{q-1}{Rq} + \frac{1}{q} \right) < 1 $$
we have $k \geq R+1$.
\end{theorem}
\begin{proof}
Just combine Proposition \ref{Prop:3.4} and Lemma \ref{Lm:2.1} as in the previous proofs.
\end{proof}
Before discussing our new method in the last section, we prove Proposition \ref{Prop:1.2}. The way is similar to the one used in the improvement of Proposition \ref{Prop:3.3}. In general we know that $q \geq 3$ and $k \geq 3$, so that:
$$ \frac{\varphi(n) \sigma(n)}{n^2} > \frac{8}{\pi^2} $$
and hence:
$$ I(n)=\frac{\sigma(n)}{n} > \frac{n}{n-1} \cdot \frac{24}{\pi^2} > \frac{24}{\pi^2}$$
If $3, 5, 7, 11, 13 \not | n$, $k \geq 4$ and:
$$ \frac{\varphi(n) \sigma(n)}{n^2} > \frac{8}{\pi^2} \cdot \frac{9}{8} \cdot \frac{25}{24} \cdot \frac{49}{48} \cdot \frac{121}{120} \cdot \frac{169}{168} $$
from which we deduce:
$$ I(n) > \frac{715715}{18432 \pi^2} $$
The conclusion follows.
\section{Discussion of our method and future research}
As our results show, our new approach via the means of group theory to Lehmer's totient problem turns out to be quite powerful. It is clear that the study of the properties of $\psi''$, and in particular the solution to the problem posed in [7] by T\u{a}rn\u{a}uceanu would probably lead to some interesting consequences related to Lehmer's problem. But we also think that other questions may be approached by our new method, namely many of the problems involving Euler totient function. We do not know if such improvements may lead to a complete to Lehmer's totient problem, though: $\psi''(C_m)$ can be really small (and in fact, $\psi''$ restricted to cyclic groups has a dense image in $[0,1]$, as shown in [11]). However, it may be possible that groups of order $4n$ with $n$ Carmichael, or maybe even other kinds of group, could share some special property for which $\psi''$ always satisfies some inequality that would guarantee that a similar result to ours holds true, implying the entire Lehmer's totient problem. In any case, we think that our idea will stimulate further research on the subject via this new group-theoretical approach.

\end{document}